\documentclass[12pt]{article}

\usepackage{amsmath}
\usepackage{amssymb}
\usepackage{amsfonts}
\usepackage{array}
\usepackage{setspace}
\usepackage{enumerate}
\usepackage{epsfig}
\usepackage{subfigure}
\usepackage{rotating}
\usepackage{color}
\usepackage{multirow}
\allowdisplaybreaks \numberwithin{equation}{section}
\begin{document}
\bibliographystyle{unsrt}
\doublespacing

\newcommand{\ve}[1]{\mbox{\boldmath$#1$}}
\newcommand{\be}{\begin{equation}}
\newcommand{\ee}{\end{equation}}
\newcommand{\bc}{\begin{center}}
\newcommand{\ec}{\end{center}}
\newcommand{\bal}{\begin{align*}}
\newcommand{\enal}{\end{align*}}
\newcommand{\al}{\alpha}
\newcommand{\bt}{\beta}
\newcommand{\gm}{\gamma}
\newcommand{\de}{\delta}
\newcommand{\la}{\lambda}
\newcommand{\om}{\omega}
\newcommand{\Om}{\Omega}
\newcommand{\Gm}{\Gamma}
\newcommand{\De}{\Delta}
\newcommand{\Th}{\Theta}
\newcommand{\nno}{\nonumber}
\newtheorem{theorem}{Theorem}[section]
\newtheorem{lemma}{Lemma}[section]
\newtheorem{assum}{Assumption}[section]
\newtheorem{claim}{Claim}[section]
\newtheorem{proposition}{Proposition}[section]
\newtheorem{corollary}{Corollary}[section]
\newtheorem{definition}{Definition}[section]
\newtheorem{remark}{Remark}[section]
\newenvironment{proof}[1][Proof]{\begin{trivlist}
\item[\hskip \labelsep {\bfseries #1}]}{\end{trivlist}}
\newenvironment{proofclaim}[1][Proof of Claim]{\begin{trivlist}
\item[\hskip \labelsep {\bfseries #1}]}{\end{trivlist}}

\def \qed {\hfill \vrule height7pt width 5pt depth 0pt}
\def\refhg{\hangindent=20pt\hangafter=1}
\def\refmark{\par\vskip 2.50mm\noindent\refhg}
\def\Bbb#1{\mbox{\sf #1}}

\title{\textbf{Data Variation in Coalescence Fractal Interpolation Function}}
\author{Srijanani Anurag Prasad}
\date{}
\maketitle \vspace{-1.5cm}
{\singlespacing \bc
Theoretical Statistics and Mathematics unit,\\
Indian Statistical Institute, Delhi Centre,\\
7, S. J. S. Sansanwal Marg, New Delhi, India. \\ Post Code:110016\\
janani@isid.ac.in
\ec }

\singlespacing

\begin{abstract}
The Iterated Function System(IFS) used in the construction of Coalescence Hidden-variable Fractal Interpolation Function depends on the interpolation data.
In this note, the effect of insertion of data on the related IFS and the Coalescence Hidden-variable Fractal Interpolation Function is studied.
\end{abstract}

{\textbf{ Key Words :} Fractal, Interpolation, Iterated Function System, Coalescence}

{Mathematics Subject Classification: Primary 28A80, 41A05 }
\newpage

\section{\sc Introduction}
The study of fractals became more important since Barnsley~\cite{barnsley86,barnsley88} introduced Fractal Interpolation Function (FIF) using the theory of Iterated Function System (IFS). FIFs provide a very effective tool for interpolation of an experimental data by a non-smooth curve. Later, Barnsley et al.~\cite{barnsley89_2} extended the idea of FIF to produce more flexible interpolation functions of a single real variable called Hidden-variable Fractal Interpolation Function (HFIF). Chand and Kapoor~\cite{chand07} constructed Coalescence Hidden-variable Fractal Interpolation Function (CHFIF) for simulation of curves that exhibit partly self-affine and partly non-self-affine nature.

In~\cite{kocic03}, Kocic discussed the problem of node insertion and knot insertion for Fractal Interpolation Function. The insertion of a new point $(\hat{x},\hat{y})$ in a given set of interpolation data is called the problem of node insertion. A knot $(\hat{x},\hat{y})$ is a node in a given set of interpolation data with a special property that $\hat{y}=f(\hat{x})$, where $f$ is FIF passing through the given set of interpolation data.  In case of CHFIF, more than one type of node insertion exists. There are four types of node insertion: Node-Node insertion, Node-Knot insertion, Knot-Node insertion and Knot-Knot insertion. In this paper, various types of node insertion and the effect of such insertions are studied.

The organization of the paper is as follows: A brief introduction on the construction of a CHFIF is given in Section~\ref{sec:two}. The problem of node insertion and various kinds of node insertion are discussed in Section~\ref{sec:three}. In Section~\ref{sec:four}, a comparative study on the smoothness of CHFIF obtained with a given set of interpolation data and on the smoothness of CHFIF obtained with a node introduced in the set of interpolation data is done. Further the bounds on fractal dimension of both the CHFIFs are also compared.

\section{Construction of CHFIF}\label{sec:two}
Let  $\Lambda = \{(x_i,y_i) \in \mathbb{R}^2 : i=0,1,\ldots,N \}$, where $-\infty <
x_0 < x_1 < \ldots < x_N < \infty$  be the given interpolation data.  A set of real parameters $\{z_i\}$ for $i=0,1,\ldots,N$ is introduced to form the generalized interpolation data  $\Delta =\{ (x_i,y_i,z_i) :i=0,1,\ldots,N \}$.  The interval $[x_0,x_N]$ is denoted by $I$.  For $n=1,2,\ldots,N$  the intervals $[x_{n-1},x_n]$ are denoted by $I_n$.  The contractive homeomorphisms   $\ L_n :I \rightarrow I_n$ for $ n = 1,2,\ldots, N $ are defined by
\begin{align}\label{eq:Ln}
L_n(x) = x_{n-1} + \frac{x_n - x_{n-1}}{x_N-x_0}\ (x - x_0)
\end{align}
For $n=1,2,\ldots, N$, the functions $ F_n : I \times \mathbb{R}^2 \rightarrow
\mathbb{R}^2 $ are defined by,
\begin{align}\label{eq:Fnxyz}
F_n(x,y,z)&=\big(\al_n y + \bt_n z + p_n(x), \gm_n z + q_n(x)\big)
\end{align}
where, $\al_n$  and $\gm_n$ are free variables chosen such that $|\al_n| < 1$ and $|\gm_n| < 1$ and $\bt_n$ are constrained variables chosen such that $|\bt_n| + |\gm_n| < 1$. The functions $p_n$ and $q_n$ are continuous functions on $x$ chosen such that the functions $F_n$ satisfy
\begin{align} \label{eq:Fncond}
F_n(x_0,y_0,z_0)=(y_{n-1},z_{n-1}) \quad \mbox{and} \quad
F_n(x_N,y_N,z_N)=(y_{n},z_{n}).
\end{align}
The above conditions are called join-up conditions. The required IFS is defined using $L_n$ and $F_n$ as
\begin{align}\label{eq:chifs}
\{I \times \mathbb{R}^2; \om_n , n=1,2,\ldots N \}
\end{align}
where,
\begin{align}\label{eq:omn}
\om _n(x,y,z)= \left( L_n(x), F_n(x,y,z)\right).
\end{align}
It has been proved in~\cite{chand07} that the above IFS is hyperbolic with respect to a metric $d^*$ on $\mathbb{R}^3,$ equivalent to the Euclidean metric. For a hyperbolic IFS, we know that there exists a unique non-empty compact set $A \subseteq \mathbb{R}^3$ such that  $ A = \bigcup\limits_{n=1}^N \omega_n(A)$. This set $A$ is called attractor of IFS for the given interpolation data and shown to be graph of a continuous function $f:I \rightarrow \mathbb{R}^2$ such that $f(x_i)=(y_i,z_i)$ for $i=0,1,\ldots, N$. Now,  Coalescence Hidden variable Fractal Interpolation Function (CHFIF) is defined as follows:

\begin{definition}
The \textbf{Coalescence Hidden-variable Fractal Interpolation
Function (CHFIF) } for the given interpolation data $\{(x_i,y_i):i =
0,1,\ldots,N \ \}$ is defined as the continuous function $\ f_1 : I
\rightarrow \mathbb{R}$ where $f_1$ is the first component of the above function $f=(f_1,f_2)$ which is graph of an atrractor.
\end{definition}

A set $S \subset \mathbb{R}^n$ consisting of points $ x = (x_1,x_2,\ldots,x_n)$ is said to be self-affine if $S$ is union of $N$ distinct subsets, each identical with $rS = \{(r_1x_1,r_2x_2,\ldots,r_nx_n) : r = (r_1,r_2,\ldots,r_n),\  r_i > 0 \ \mbox{and} \ x \in S\}$ up to translation and rotation. If $S$ is not self-affine, then it is non-self-affine. In a self-affine set, suppose $r_1 = r_2= \ldots =r_n$. Then, it is called self-similar set, otherwise it is called non-self-similar set.

\begin{remark}
The function $f_1$ is called a CHFIF as it exhibits both self-affine and non-self-affine nature. For the same interpolation data, the function $f_2$ is a self-affine function.
\end{remark}

\begin{remark}
For a given an interpolation data with $N+1$ points, there are $N$ contraction mappings defined in the IFS.
\end{remark}

\section{Node Insertion}\label{sec:three}
Given an interpolation data $\Lambda =\{(x_0,y_0), (x_1,y_1),\ldots ,(x_N,y_N)\}$, the insertion of a new point $(\hat{x},\hat{y})$ in the given interpolation data is called the problem of node insertion. Suppose $x_{k-1} < \hat{x} < x_k$. Then, the new interpolation data is $\{(x_0,y_0), (x_1,y_1),\ldots ,\\(x_{k-1},y_{k-1}), (\hat{x},\hat{y}),(x_{k},y_{k}), \ldots,(x_N,y_N)\}$ consists of $N+2$ points. We need $N+1$ contraction maps to define IFS corresponding to these data.

The interval $I_k=[x_{k-1},x_k]$ is broken into two smaller intervals $I_k^l=[x_{k-1},\hat{x}]$ and $I_k^r=[\hat{x},x_k]$.
Define $L_k^l : I \rightarrow I_k^l$ and $L_k^r : I \rightarrow I_k^r$ as
\begin{align}\label{eq:Lklr}\left.
\begin{array}{rl}
L_k^l(x) &= x_{k-1} + \frac{\hat{x} - x_{k-1}}{x_N-x_0}\ (x - x_0) \\
L_k^r(x) &= \hat{x} + \frac{x_k-\hat{x}}{x_N-x_0}\ (x - x_0)
\end{array} \right\}
\end{align}
Similarly, define $F_k^l : I \times \mathbb{R}^2 \rightarrow \mathbb{R}^2$ and
$F_k^r : I \times \mathbb{R}^2 \rightarrow \mathbb{R}^2$ by
\begin{align}\label{eq:Fklr}\left.
\begin{array}{rl}
F_k^l(x) &= \big(\al_k^l y + \bt_k^l z + p_k^l(x), \gm_k^l z + q_k^l(x)\big)\\
F_k^r(x) &= \big(\al_k^r y + \bt_k^r z + p_k^r(x), \gm_k^r z + q_k^r(x)\big)
\end{array} \right\}
\end{align}
where, $\al_k^l, \ \al_k^r,\ \gm_k^l$ and $\gm_k^r$ are free variables chosen such that $|\al_k^l| < 1$, $|\al_k^r| < 1$, $|\gm_k^l| < 1$ and $|\gm_k^r| < 1$. The variables $|\bt_k^l| $ and $|\bt_k^r| $ are constrained variables chosen such that $|\bt_k^l| + |\gm_k^l| < 1$ and $|\bt_k^r| + |\gm_k^r| < 1$. The functions $p_k^l$, $p_k^r$, $q_k^l$ and $q_k^r$ are continuous functions on $x$ chosen such that the functions $F_k^l$ and $F_k^r$  satisfy
\begin{align} \label{eq:Fklrcond}
\left. \begin{array}{rlrl}
F_k^l(x_0,y_0,z_0)&=(y_{n-1},z_{n-1}) \quad &\mbox{and} \quad
F_k^l(x_N,y_N,z_N)&=(\hat{y},\hat{z})  \\
F_k^r(x_0,y_0,z_0)&=(\hat{y},\hat{z}) \quad &\mbox{and} \quad
F_k^r(x_N,y_N,z_N)&= (y_n,z_n)
\end{array} \right\}
\end{align}
where, $\hat{z}$ is any real parameter such that the generalized interpolation data is \\ $\widehat{\Delta}=\{(x_0,y_0,z_0), (x_1,y_1,z_1),\ldots ,(x_{k-1},y_{k-1},z_{k-1}), (\hat{x},\hat{y},\hat{z}),(x_k,y_k,z_k), \ldots,(x_N,y_N,z_N)\}$.  This is called Node-Node insertion problem. If $\hat{z} = f_2(\hat{x})$ but $\hat{y} \neq f_1(\hat{x})$, then it is called Node-Knot insertion problem. If $\hat{y} = f_1(\hat{x})$ but $\hat{z} \neq f_2(\hat{x})$, then it is called Knot-Node insertion problem. And if $\hat{z} = f_2(\hat{x})$ and $\hat{y} = f_1(\hat{x})$, then it is called Knot-Knot insertion problem. The Node-Knot, Knot-Node and Knot-Knot insertion problem are special cases of Node-Node insertion problem. Hence, the following theorems hold for all cases.

\begin{theorem}
Let $\widehat{\Lambda} = \Lambda \bigcup (\hat{x},\hat{y})$ such that $x_{k-1} < \hat{x} < x_k$. Then, \begin{align}\label{eq:newchifs}
\{I \times \mathbb{R}^2; \om_n , n=1,2,\ldots N, n\neq k, \om_k^l,\om_k^r \}
\end{align}
with $\om_k^l =(L_k^l,F_k^l)$ and $\om_k^r =(L_k^r,F_k^r)$  is a hyperbolic IFS on  $I \times \mathbb{R}^2$ and there exists an attractor $\hat{A}$ satisfying $\hat{A} = \bigcup\limits_{\substack{n=1 \\ n\neq k}}^N \omega_n(\hat{A})\bigcup\omega_k^l(\hat{A})\bigcup\omega_k^r(\hat{A})$.
\end{theorem}
\begin{proof}
Suppose $\al_k^l \leq \al_k$, $\gm_k^l \leq \gm_k$, $\bt_k^l \leq \bt_k$, $\al_k^r \leq \al_r$, $\gm_k^r \leq \gm_k$ and $\bt_k^r \leq \bt_k$.
 Then the maps $\om_k^l$ and $\om_k^r$ are contraction maps with respect to same metric by which $\om_n$ are contraction maps. Otherwise, a metric could be defined as in~\cite{chand07} such that $\om_n , n=1,2,\ldots N, n\neq k, \om_k^l,\om_k^r$ are contraction maps. Therefore, the IFS given by~\eqref{eq:newchifs}is hyperbolic and has an attractor $\hat{A}$ satisfying $\hat{A} = \bigcup\limits_{\substack{n=1 \\ n\neq k}}^N \omega_n(\hat{A})\bigcup\omega_k^l(\hat{A})\bigcup\omega_k^r(\hat{A})$.
\qed \end{proof}

\begin{theorem}
The attractor of the IFS given by~\eqref{eq:newchifs} is graph of a continuous function passing through the generalized interpolation points $\widehat{\Delta} = \Delta \bigcup (\hat{x},\hat{y},\hat{z})$.
\end{theorem}
\begin{proof}
Consider the metric space of continuous functions $ (\mathcal{G}, d_{\mathcal{G}})$ such that $ \mathcal{G} = \{ g: \  g : I \rightarrow \mathbb{R}^2 \ \mbox{is continuous},
g(x_0) = (y_0,z_0) \ \mbox{and} \ g(x_N) = (y_N,z_N) \}$ and $ d_{\mathcal{G}} (g,\hat{g}) = \max\limits_{x \in I }( |g_1(x) - \hat{g}_1(x)|, |g_2(x)-\hat{g}_2(x)|)$,   $g,\hat{g} \in \mathcal{G}$. Define Read-Bajraktarevic operator on the above space to construct CHFIF passing through the interpolation data $\widehat{\Lambda}$ as
\begin{align*}
\left. \begin{array}{rl}
\hat{T}(g)(x)&= F_n(L_n^{-1}(x),g(L_n^{-1}(x))),   \ x \in I_n, n \neq k \\
\hat{T}(g)(x)&= F_k^l(L_k^{l\ -1}(x),g(L_k^{l\ -1}(x))), \ x \in I_k^l \\
\hat{T}(g)(x)&= F_k^l(L_k^{l\ -1}(x),g(L_k^{l\ -1}(x))), \ x \in I_k^l
\end{array} \right \}
\end{align*}

Following the lines of proof as in~\cite{barnsley89_2}, it can be easily shown that the Read-Bajraktarevic operator is a contraction map and there exist a continuous function $\hat{f}: I \rightarrow \mathbb{R}^2$ passing through the generalized interpolation points $\widehat{\Delta}$. Also, uniqueness gives that $\hat{A}$ is graph of the function $\hat{f}$.
\qed \end{proof}

\begin{remark}
Like earlier, the function $\hat{f}$ is expressed component wise as $\hat{f}=(\hat{f}_1,\hat{f}_2)$. Then, $\hat{f}_1$ is a CHFIF passing through the interpolation data $\widehat{\Lambda}$.
\end{remark}

\begin{remark}
Denote $\rho_x = \frac{\hat{x}-x_{k-1}}{x_k-x_{k-1}}$. The contractive homeomorphisms $L_k^l$ and $L_k^r$ could be expressed using the contractive homeomorphism $L_k$ as follows:
\begin{align*}
L_k^l(x)&= \rho_x \ L_k(x) + (1-\rho_x) x_{k-1} \\
L_k^r(x)&= (1-\rho_x) \ L_k(x) + \rho_x x_k
\end{align*}
\end{remark}

\begin{remark}
 Suppose $\hat{z} = f_2(\hat{x})$ but $\hat{y} \neq f_1(\hat{x})$. Since $(\hat{x},\hat{z})$ is a knot on the interpolation data $\{(x_i,z_i)i=0,1,\ldots,N\}$, it is shown in~\cite{kocic03} that $\hat{f}_2=f_2$. Hence, it is called Node-Knot insertion problem.
\end{remark}

\begin{remark}
 In case of Knot-Node insertion, since $\hat{f}_1$ depends on $\hat{f}_2$,  $\hat{f}_1(x) \neq f_1(x) $ if $x \neq x_i, i=0,1,\ldots,N$ and $x \neq \hat{x}$.
\end{remark}

Replacing $Y^0$ in~\cite{kocic03} by $\Delta^0=\Delta =\{ (x_i,y_i,z_i) :i=0,1,\ldots,N \}$ and $Y^j$ by $\Delta^j=\bigcup\limits_{i(j) \in \sum_j} \om_{i(j)}(\Delta^0)$, where $i(j)=(i_1,i_2,\ldots,i_j)$ are finite codes of length $j$ and is an element of the set $\sum_{j} = \{1,2,\ldots,N\}^{1,\ldots,j}$, it is clear that $f:I \rightarrow \mathbb{R}^2$ which is a fixed point of the Read-Bajraktarevic opearator $T(g)=F_{n}(L_{n}^{-1}(\cdot), g(L_{n}^{-1}(\cdot)))$ defined on the space $C^*(\Delta)=\{g: g:I \rightarrow \mathbb{R}^2 \ \mbox{such that}\ g(x_i)=(y_i,z_i) \}$ also interpolates $\Delta^j$. Consider the set $\Lambda^j$ which consists of points $(x_n,y_n)$ if $(x_n,y_n,z_n) \in \Delta^j$. Then, $f_1:I \rightarrow \mathbb{R}$  interpolates $\Lambda^j$. Similar to Lemma $3.1$ and $3.2$ of~\cite{kocic03}, we have the following proposition:

\begin{proposition}
(1): $\Delta^{j-1} \subset \Delta^j$ \\
(2): The IFSs $\{\mathbb{R}^3, \om_i; i=1,2,\ldots,N\}$ and $\{\mathbb{R}^3, \om_{i(j)}; i=1,2,\ldots,N^j\}$, $j\geq 2$, have the same attractor $A$.
\end{proposition}
\begin{proof}
Proof is similar to Lemma $3.1$ and $3.2$ in~\cite{kocic03} if we replace $Y^j$ by $\Delta^j$ and use $\om_i$ defined by~\eqref{eq:omn}.
\qed \end{proof}

Define $C^*(\Delta^j)=\{g: g:I \rightarrow \mathbb{R}^2 \ \mbox{which interpolates} \Delta^j\}$ and the Read-Bajraktarevic operator on the space  as
\begin{align*}
T_j(g)(x)=F_{i(j)}(L_{i(j)}^{-1}(x), g(L_{i(j)}^{-1}(x)))\ x \in [L_{i(j)}(x_0),L_{i(j)}(x_N)].
\end{align*}
It is clear that the $T_j$ is a contraction map on $C^*(\Delta^j)$ with respect to maximum metric. Using the above proposition, the fixed point of $T_j$ say $f^j$ interpolates $\Delta^n$ for all $n < j$.

\begin{theorem}
Given an interpolation data $\Lambda =\{(x_0,y_0), (x_1,y_1),\ldots ,(x_N,y_N)\}$, let $(\hat{x},\hat{y}, \hat{z})$ be a knot in the generalized interpolation data $\Delta=\{(x_0,y_0,z_0), (x_1,y_1,z_1),\ldots ,\\(x_N,y_N,z_N)\}$ such that $\hat{y}=f_1(\hat{x})$ and $\hat{z}=f_2(\hat{z})$. Then,
\begin{align*}
f_1(L_{i(2)}(x))&=\al_{i_2} (\al_{i_1} f_1(x)+\bt{i_1}f_2(x)+p_{i_1}(x)) + \bt_{i_2} (\gm_{i_1} f_2(x)+q_{i_1}(x)) + p_{i_2}(L_{i_1}(x)) \\
f_2(L_{i(2)}(x))&= \gm_{i_2}( \gm_{i_1} f_2(x)+q_{i_1}(x))+q_{i_2}(L_{i_1}(x))
\end{align*}
where, $i(2)=(i_1,i_2) \in \sum_2$.
\end{theorem}

\begin{proof}
Consider $T_2(g)(x)=F_{i(2)}(L_{i(2)}^{-1}(x), g(L_{i(2)}^{-1}(x))),\ \ i(2)=(i_1,i_2) \in \sum_2 $  for $x \in [L_{i(2)}(x_0),L_{i(2)}(x_N)]= [L_{i_2}(L_{i_1}(x_0)),L_{i_2}(L_{i_1}(x_N))]$. Then, $T_2$ is a contraction map on $C^*(\Delta^2)$ with respect to maximum metric and  $h$ be the fixed point of $T_2$. Since $F_{i(2)}(x,h(x))=F_{i_1i_2}(x,h(x))= F_{i_2}(L_{i_1}(x),F_{i_1}(x,h(x)))$, it is seen that,
\begin{align*}
h(x)&= F_{i(2)}(L_{i(2)}^{-1}(x), h(L_{i(2)}^{-1}(x))) \nno \\ &=F_{i_1i_2}(L_{i_1}^{-1}(L_{i_2}^{-1}(x)), h(L_{i_1}^{-1}(L_{i_2}^{-1}(x)))) \nno \\
&=F_{i_2}(L_{i_1}(L_{i_1}^{-1}(L_{i_2}^{-1}(x))), F_{i_1}(L_{i_1}^{-1}(L_{i_2}^{-1}(x)),h(L_{i_1}^{-1}(L_{i_2}^{-1}(x))))) \nno \\
&=F_{i_2}(L_{i_2}^{-1}(x), F_{i_1}(L_{i_1}^{-1}(L_{i_2}^{-1}(x)),h(L_{i_1}^{-1}(L_{i_2}^{-1}(x))))).
\end{align*}
Also, we know that $f$ satisfies the functional equation $f(x)=F_i(L_{i}^{-1}(x), f(L_{i}^{-1}(x)))\ $ for $\ x \in [L_{i}(x_0),L_{i}(x_N)], \ i \in \{1,\ldots,N\}$. Hence,
\begin{align*}
f(x)= F_{i_2}(L_{i_2}^{-1}(x), f(L_{i_2}^{-1}(x))) = F_{i_2}(L_{i_2}^{-1}(x), F_{i_1}(L_{i_1}^{-1}(L_{i_2}^{-1}(x)),f(L_{i_1}^{-1}(L_{i_2}^{-1}(x)))))
\end{align*}
which implies $f$ is a fixed point of $T_2$. By uniqueness, we have $f=h$.

\newpage
\noindent Now, $f(x)=(f_1(x),f_2(x))$ in the above equation, we have
\begin{align*}
(f_1(L_{i_2}& (L_{i_1}(x))), f_2(L_{i_2}(L_{i_1}(x)))) \\ &=F_{i_2}(L_{i_1}(x),F_{i_1}(x,f_1(x),f_2(x)))\\
&=F_{i_2}(L_{i_1}(x),\al_{i_1}f_1(x)+\bt{i_1}f_2(x)+p_{i_1}(x), \gm_{i_1} f_2(x)+q_{i_1}(x))\\
&=\bigg(\al_{i_2} (\al_{i_1}f_1(x)+\bt{i_1}f_2(x)+p_{i_1}(x)) + \bt_{i_2} (\gm_{i_1} f_2(x)+q_{i_1}(x)) + p_{i_2}(L_{i_1}(x)), \\ & \quad \gm_{i_2}( \gm_{i_1} f_2(x)+q_{i_1}(x))+q_{i_2}(L_{i_1}(x)\bigg)
\end{align*}
which implies $f_1(L_{i(2)}(x))=\al_{i_2} (\al_{i_1}f_1(x)+\bt_{i_1}f_2(x)+p_{i_1}(x)) + \bt_{i_2} (\gm_{i_1} f_2(x)+q_{i_1}(x)) + p_{i_2}(L_{i_1}(x))$  and $f_2(L_{i(2)}(x)) = \gm_{i_2}( \gm_{i_1} f_2(x)+q_{i_1}(x))+q_{i_2}(L_{i_1}(x))$.
\qed \end{proof}

\section{Smoothness}\label{sec:four}
In this section, we compare the Lipschitz exponent of  the functions $\hat{f}_1$ passing through the interpolation data $\widehat{\Lambda}$ and $f_1$ passing through the interpolation data $\Lambda$. Through out this section, we assume $0=x_0 < x_1 < \ldots < x_N=1$.

A function $f:\mathbb{R} \rightarrow \mathbb{R}$ is said to be Lipschitz function of order $\delta$ if $|f(x)-f(\bar{x})| \leq K |x-\bar{x}|^{\delta}$, where $K$ is any positive constant and $0< \delta \leq 1$. The Modulus of continuity of a function $f$ is given by $\om(f; t) = \sup\limits_{|h|<t} \sup\limits_x |f(x + h) - f(x)|$.

Consider the special case when the functions $p_n$, $q_n$, $p_k^l$, $p_k^r$, $q_k^l$ and $q_k^r$ are linear functions. Denoting $\rho_y = \frac{\hat{y}-y_{k-1}}{y_k-y_{k-1}}$ and $\rho_z =\frac{\hat{z}-z_{k-1}}{z_k-z_{k-1}}$, it is observed that,
\begin{align}\label{eq:rel}
q_k^l(x) &= \rho_z\ q_k(x) + (1-\rho_z)\ z_{k-1} + (\rho_z\ \gm_k - \gm_k^l)\left(\frac{z_N(x-x_0)+z_0(x_N-x)}{(x_N-x_0)}\right) \nno \\
q_k^r(x) &= (1-\rho_z)\ q_k(x) + \rho_z\ z_k + [(1-\rho_z)\ \gm_k - \gm_k^r]\left(\frac{z_N(x-x_0)+z_0(x_N-x)}{(x_N-x_0)}\right) \nno \\
p_k^l(x) &= \rho_y\ p_k(x) + (1-\rho_y)\ y_{k-1} \nno \\ & \quad \mbox{}+ [\rho_y \al_k - \al_k^l] \left(\frac{y_N(x-x_0)+y_0(x_N-x)}{(x_N-x_0)}\right) \nno \\ & \quad \mbox{}+[\rho_z \bt_k - \bt_k^l]\left(\frac{(z_N(x-x_0)+z_0(x_N-x)}{(x_N-x_0)}\right) \nno \\
\mbox{and} \quad p_k^r(x) &= (1-\rho_y)\ p_k(x) + \rho_y\ y_k \nno \\ & \quad \mbox{}+ [(1-\rho_y) \al_k - \al_k^r]\left(\frac{y_N(x-x_0)+y_0(x_N-x)}{(x_N-x_0)}\right)
\nno \\ & \quad \mbox{}+ [(1-\rho_z) \bt_k - \bt_k^r]\left(\frac{z_N(x-x_0)+z_0(x_N-x)}{(x_N-x_0)}\right)
\end{align}
We use the above relation~\eqref{eq:rel} between $p_n$, $q_n$, $p_k^l$, $p_k^r$, $q_k^l$ and $q_k^r$ even when these functions are not linear polynomial and find Lipschitz exponent of the functions $p_k^l$, $p_k^r$, $q_k^l$ and $q_k^r$ in the following proposition:

\begin{proposition}
Let $x_0=0$, $x_N=1$,  $p_k \in \mbox{Lip}\ \lambda_k$ and $q_k \in \mbox{Lip}\ \mu_k$. Then, $p_k^l, p_k^r \in \mbox{Lip}\ \lambda_k$ and $q_k^l, q_k^r \in \mbox{Lip}\ \mu_k$.
\end{proposition}
\begin{proof}
By~\eqref{eq:rel}, it is observed that
\begin{align}\label{eq:qkl}
|q_k^l(x)-q_k^l(\bar{x})| & \leq |\rho_z|\ |q_k(x)-q_k(\bar{x})|  + |\rho_z\ \gm_k - \gm_k^l|\bigg((|z_N|+ |z_0|)|x-\bar{x}|\bigg) \nno \\
& \leq |\rho_z|\ K_1 |x-\bar{x}|^{\mu_k}  + |\rho_z\ \gm_k - \gm_k^l|\bigg((|z_N|+ |z_0|)|x-\bar{x}|\bigg) \nno \\
& \leq M_1 |x-\bar{x}|^{\lambda_k}
\end{align}
and
\begin{align}\label{eq:qkr}
|q_k^r(x)-q_k^r(\bar{x})| & \leq |1-\rho_z|\ |q_k(x)-q_k(\bar{x})| + |(1-\rho_z)\ \gm_k - \gm_k^r|\bigg((|z_N|+ |z_0|)|x-\bar{x}|\bigg) \nno \\
& \leq |1-\rho_z|\ K_1 |x-\bar{x}|^{\mu_k}  + |(1-\rho_z)\ \gm_k - \gm_k^r| \bigg((|z_N|+ |z_0|)|x-\bar{x}|\bigg) \nno \\
& \leq M_2 |x-\bar{x}|^{\lambda_k}
\end{align}
where, $K_1$ is a positive constant such that $ |q_k(x)-q_k(\bar{x})| \leq K_1 |x-\bar{x}|^{\mu_k}$,  $M_1$ and $M_2$ are positive constants given by $M_1 = \max(|\rho_z| K_1 ,|\rho_z \gm_k - \gm_k^l|(|z_N|+ |z_0|))$ and $M_2 = \max(|1-\rho_z| K_1 ,|\rho_z(1-\gm_k) - \gm_k^l|(|z_N|+ |z_0|))$. Hence the functions $q_k^l$  and $ q_k^r \in \mbox{Lip}\ \mu_k$. Similarly,
\begin{align}\label{eq:pkl}
|p_k^l(x)-p_k^l(\bar{x})| 
& \leq  |\rho_y|\ K_2  |x-\bar{x}|^{\lambda_k}  + |\rho_y \al_k - \al_k^l| \bigg((|y_N|+ |y_0|)|x-\bar{x}|\bigg) \nno \\ & \quad \mbox{} +|\rho_z \bt_k - \bt_k^l|\bigg((|z_N|+ |z_0|)|x-\bar{x}|\bigg) \nno \\
& \leq M_3 |x-\bar{x}|^{\lambda_k}
\end{align}
and
\begin{align}\label{eq:pkr}
|p_k^r(x)-p_k^r(\bar{x})| 
& \leq  |1-\rho_y|\ K_2 |x-\bar{x}|^{\lambda_k} + + |(1-\rho_y) \al_k - \al_k^r| \bigg((|y_N|+ |y_0|)|x-\bar{x}|\bigg) \nno \\ & \quad \mbox{} + |(1-\rho_z) \bt_k - \bt_k^r| \bigg((|z_N|+ |z_0|)|x-\bar{x}|\bigg) \nno \\
& \leq M_4 |x-\bar{x}|^{\lambda_k}
\end{align}
where, $K_2$ is a positive constant such that $ |p_k(x)-p_k(\bar{x})| \leq K_2 |x-\bar{x}|^{\la_k}$,  $M_3$ and $M_4$ are positive constants given by $M_3 = \max(|\rho_y| K_2,|\rho_y \al_k - \al_k^l|(|y_N|+ |y_0|),|\rho_z \bt_k - \bt_k^l|(|z_N|+ |z_0|))$ and $M_4 = \max(|1-\rho_y| K_2,|(1-\rho_y) \al_k - \al_k^l|(|y_N|+ |y_0|),|(1-\rho_z) \bt_k - \bt_k^l|(|z_N|+ |z_0|))$.
 From~\eqref{eq:pkl} and \eqref{eq:pkr}, it is clear that $p_k^l,p_k^r \in \mbox{Lip}\ \lambda_k$.
 \qed \end{proof}

Let $ \la , \mu , \al, \gm, \bt, \Om_n, \Gm_n,  \Th_n,  \Om, \Gm$ and $\Th$ be as defined in~\cite{chand07}. The smoothness of a CHFIF was classified according to the values of $\Om, \Gm$ and $\Th$. With the help of above proposition, we define
\begin{align}\label{eq:OmGmThlr}
\Om_k^l &= \frac{|\al_k^l|}{|I_k^l|^{\la}}, \Gm_k^l = \frac{|\gm_k^l|}{|I_k^l|^{\mu}},\Th_k^l = \frac{|\al_k^l|}{|I_k^l|^{\mu}} \nno \\
\Om_k^r &=  \frac{|\al_k^r|}{|I_k^r|^{\la}}, \Gm_k^r = \frac{|\gm_k^r|}{|I_k^r|^{\mu}},\Th_k^r = \frac{|\al_k^r|}{|I_k^r|^{\mu}} \nno \\
\widehat{\Om}&= \max \{\Om_n : n = 1,2,\ldots,N, n \neq k , \Om_k^l, \Om_k^r \} \nno \\
\widehat{\Gm} &= \max \{\Gm_n : n = 1,2,\ldots,N, n \neq k, \Gm_k^l, \Gm_k^r \} \nno \\
\widehat{\Th} &= \max \{ \Th_n : n = 1,2,\ldots,N, n \neq k \Th_k^l, \Th_k^r \}.
\end{align}

The following theorem gives the relation between $\Om$ and $\widehat{\Om}$.
\begin{theorem}\label{th:1}
Let $x_0=0$, $x_N=1$, $\al_k^l =\rho_x \al_k$, $\al_k^r =(1-\rho_x) \al_k$  and $p_n \in \mbox{Lip}\ \lambda_n$. Then,
\begin{enumerate}
\item $\widehat{\Om} < 1$  if  $\Om < 1$ or $ \Om_k < \frac{1}{\max(|\rho_x|^{1-\la},|1-\rho_x|^{1-\la})}$ and $\Om_n < 1, n \neq k$
\item $\widehat{\Om} = 1$  if $ \Om = \Om_n = 1$ for some $n \neq k$ or $1 < \Om =\Om_k = \frac{1}{\max(|\rho_x|^{1-\la},|1-\rho_x|^{1-\la})} $ and $\Om_n \leq 1$ for all $n \neq k$.
\item $\widehat{\Om} > 1$ if $\Om = \Om_n > 1 $ for some $n \neq k$ or $\Om = \Om_k > \frac{1}{\max(|\rho_x|^{1-\la},|1-\rho_x|^{1-\la})} > 1$.
\end{enumerate}
\end{theorem}
\begin{proof}
Using $\al_k^l =\rho_x \al_k$ and $\al_k^r =(1-\rho_x) \al_k$ in $\Om_k^l$ and $\Om_k^r$, it is observed that
\begin{align}\label{eq:omklr} \left.
\begin{array}{rl}
\Om_k^l &= \frac{|\al_k^l|}{|I_k^l|^{\la}} = \frac{|\rho_x \al_k|}{|\rho_x I_k|^{\la}} =  |\rho_x|^{1-\la} \Om_k \\
\Om_k^r &= \frac{|\al_k^r|}{|I_k^r|^{\la}} = \frac{|(1-\rho_x) \al_k|}{|(1-\rho_x) I_k|^{\la}} =  |1-\rho_x|^{1-\la} \Om_k .
\end{array} \right\}
\end{align}
Hence, $\Om_k^l, \Om_k^r  < \Om_k $   implies $ \widehat{\Om} \leq \Om $

Case 1: \textbf{$\Om < 1 $} \\
The above inequality gives $\widehat{\Om} \leq \Om < 1$.

Case 2: \textbf{$\Om = 1 $} \\
If  $\Om = \Om_n = 1 $ for some $ n \neq k$ then $\widehat{\Om} = 1$.

Suppose $\Om = \Om_k = 1 $ and $\Om_n <1$ for all $ n \neq k$. Then~\eqref{eq:omklr} gives $\widehat{\Om} < 1$.

Case 3. \textbf{$\Om > 1 $} \\
If  $\Om = \Om_n > 1 $ for some $ n \neq k$ then $\widehat{\Om} > 1$.

Let $\Om = \Om_k > 1 $ and $\Om_n \leq 1$ for all $ n \neq k$. Then $\widehat{\Om} > 1$ if $\Om_k > \frac{1}{\max(|\rho_x|^{1-\la},|\rho_x|^{1-\la})}$. If $\Om_n = 1$ for some $ n \neq k$ or $\Om = \Om_k =  \frac{1}{\max(|\rho_x|^{1-\la},|\rho_x|^{1-\la})}$ then $\widehat{\Om}  = 1$. Finally, if $\Om_n <1$ for all $ n \neq k$ and $ \Om = \Om_k < \frac{1}{\max(|\rho_x|^{1-\la},|\rho_x|^{1-\la})}$ then $\widehat{\Om}  < 1$.
\qed \end{proof}

The next theorem gives the relation between $\Gm$ and $\widehat{\Gm}$.
\begin{theorem}\label{th:2}
Let $x_0=0$, $x_N=1$, $\gm_k^l =\rho_x \gm_k$, $\gm_k^r =(1-\rho_x) \gm_k$  and $q_n \in \mbox{Lip}\ \mu_n$. Then,
\begin{enumerate}
\item $\widehat{\Gm} < 1$  if  $\Gm < 1$ or $ \Gm_k < \frac{1}{\max(|\rho_x|^{1-\mu},|1-\rho_x|^{1-\mu})}$ and $\Gm_n < 1, n \neq k$
\item $\widehat{\Gm} = 1$  if $ \Gm = \Gm_n = 1$ for some $n \neq k$ or $1 < \Gm =\Gm_k = \frac{1}{\max(|\rho_x|^{1-\mu},|1-\rho_x|^{1-\mu})} $ and $\Gm_n \leq 1$ for all $n \neq k$.
\item $\widehat{\Gm} > 1$ if $\Gm = \Gm_n > 1 $ for some $n \neq k$ or $\Gm = \Gm_k > \frac{1}{\max(|\rho_x|^{1-\mu},|1-\rho_x|^{1-\mu})} > 1$.
\end{enumerate}
\end{theorem}
\begin{proof}
The proof is similar to Theorem~\ref{th:1} by replacing $\al$ by $\gm$, $\al_k^l$ by $\gm_k^l$, $\al_k^r$ by $\gm_k^r$,  $\Om$ by $\Gm$ and $\widehat{\Om} $ by $\widehat{\Gm}$ and hence omitted.
\qed \end{proof}

Similarly, the relation between $\Th$ and $\widehat{\Th}$ is obtained as follows:
\begin{theorem}\label{th:3}
Let $x_0=0$, $x_N=1$, $\al_k^l =\rho_x \al_k$, $\al_k^r =(1-\rho_x) \al_k$  and $q_n \in \mbox{Lip}\ \mu_n$. Then,
\begin{enumerate}
\item $\widehat{\Th} < 1$  if  $\Th < 1$ or $ \Th_k < \frac{1}{\max(|\rho_x|^{1-\mu},|1-\rho_x|^{1-\mu})}$ and $\Th_n < 1, n \neq k$
\item $\widehat{\Th} = 1$  if $ \Th = \Th_n = 1$ for some $n \neq k$ or $1 < \Th =\Th_k = \frac{1}{\max(|\rho_x|^{1-\mu},|1-\rho_x|^{1-\mu})} $ and $\Th_n \leq 1$ for all $n \neq k$.
\item $\widehat{\Th} > 1$ if $\Th = \Th_n > 1 $ for some $n \neq k$ or $\Th = \Th_k > \frac{1}{\max(|\rho_x|^{1-\mu},|1-\rho_x|^{1-\mu})} > 1$.
\end{enumerate}
\end{theorem}
\begin{proof}
The proof is similar to Theorem~\ref{th:1} by replacing $\la$ by $\mu$,   $\Om$ by $\Th$ and $\widehat{\Om} $ by $\widehat{\Th}$ and hence omitted.
\qed \end{proof}

\begin{remark}
If $\la=\mu=1$ then $\widehat{\Th}=\widehat{\Om}=\Om=\Th$ and $\widehat{\Gm}=\Gm$.
\end{remark}

Similar to Theorems $3.1-3.3$ in~\cite{chand07}, the smoothness of CHFIF  $\hat{f}_1$ passing through the interpolation data $\widehat{\Lambda}$ can be obtained as follows:
\begin{enumerate}
\item  $\ \hat{f}_1 \in \mbox{Lip} \ \hat{\de} $ if $\widehat{\Th} \neq 1$, $\widehat{\Om} \neq 1$ and $\widehat{\Gm} \neq 1$,

\item  $\om(\hat{f}_1; t)  = O(|t|^{\hat{\de}} \log |t|)$ if  $\widehat{\Th} \neq 1$,  $\widehat{\Om} = 1$ or $\widehat{\Gm} = 1$,

\item  $\om(\hat{f}_1; t)  = O(|t|^{\hat{\de}} \log |t|)$ if $\widehat{\Th} = 1$,  $\widehat{\Gm} \neq 1$ and $\widehat{\Om} \in \mathbb{R}$,

\item $\om(\hat{f}_1; t)  = O(|t|^{\hat{\de}} \log |t|^2)$ if $\widehat{\Th} = 1$,  $\widehat{\Gm} = 1$ and $\widehat{\Om} \in \mathbb{R}$.
\end{enumerate}
Theorems~\ref{th:1}-~\ref{th:3} helps to compare the Lipschitz exponent of CHFIF $\hat{f}_1$ passing through the interpolation data $\widehat{\Lambda}$ and CHFIF $f_1$ passing through the interpolation data $\Lambda$.

\begin{theorem}\label{th:sm}
Let $x_0=0$, $x_N=1$, $\gm_k^l =\rho_x \gm_k$, $\gm_k^r =(1-\rho_x) \gm_k$, $\al_k^l =\rho_x \al_k$, $\al_k^r =(1-\rho_x) \al_k$, $\bt_k^l =\rho_x \bt_k$, $\bt_k^r =(1-\rho_x) \bt_k$, $p_n \in \mbox{Lip}\ \lambda_n$ and $q_n \in \mbox{Lip}\ \mu_n$. Also, let $f_1$ and $\hat{f}_1$ be the CHFIFs passing through $\Lambda$ and $\widehat{\Lambda}$ respectively. If $\Om  \leq 1$, $\Gm  \leq 1$ and $\Th  \leq 1$, then one of the following is true:
\begin{enumerate}
\item $f_1$ and $\hat{f}_1$ belongs to same Lipschitz class
\item $f_1$ and $\hat{f}_1$ have same modulus of continuity i.e. $\om(\hat{f}_1; t) = \om(f_1; t) = O(|t|^{\de} \log |t|)$.
\item $\hat{f}_1 \in \mbox{Lip}\  \de$ while $\om(f_1; t) = O(|t|^{\de} \log |t|)$
\end{enumerate}
\end{theorem}
\begin{proof}
\begin{enumerate}
\item From Theorems~\ref{th:1}-~\ref{th:3}, it is seen that if $\Om < 1$, $\Gm < 1$ and $\Th < 1$ then $\widehat{\Om} < 1$, $\widehat{\Gm} < 1$ and $\widehat{\Th} < 1$. It is shown in~\cite{chand07} that $f_1 \in \mbox{Lip}\ \de$, where $\de = \min(\la, \mu)$. Following the lines of proof of Theorem $3.1$ in~\cite{chand07}, it is easily proved that $\hat{\de} = \min(\la, \mu) = \de$.

\item If $\Om = \Om_{n1} =1$, $\Gm = \Gm_{n2} =1$ and $\Th = \Th_{n3}=1$ for some $n1,n2,n3 \neq k$, then $\widehat{\Om} = \Om = 1$, $\widehat{\Gm} = \Gm =1$ and $\widehat{\Th} = \Th =1$. Hence, $ \om(f_1; t) = O(|t|^{\de} \log |t|)$ and $\om(\hat{f}_1; t) = O(|t|^{\hat{\de}} \log |t|)$. Since $\hat{\de} = \min(\la, \mu) = \de$, $f_1$ and $\hat{f}_1$ have same modulus of continuity i.e. $\om(\hat{f}_1; t) = \om(f_1; t) = O(|t|^{\de} \log |t|)$.

\item If $\Om = \Om_k =1$, $\Gm = \Gm_k =1$ and $\Th = \Th_k=1$, then $\widehat{\Om} < 1$, $\widehat{\Gm} <1$ and $\widehat{\Th}<1$. So, $\hat{f}_1 \in \mbox{Lip}\  \hat{\de} = \mbox{Lip}\ \de $ while $\om(f_1; t) = O(|t|^{\de} \log |t|)$.
\end{enumerate}
\qed
\end{proof}

The above theorem helps in comparing the bounds of fractal dimension of CHFIF $\hat{f}_1$ with that of CHFIF $f_1$.

\begin{theorem}
Let $x_0=0$, $x_N=1$, $\gm_k^l =\rho_x \gm_k$, $\gm_k^r =(1-\rho_x) \gm_k$, $\al_k^l =\rho_x \al_k$, $\al_k^r =(1-\rho_x) \al_k$, $\bt_k^l =\rho_x \bt_k$, $\bt_k^r =(1-\rho_x) \bt_k$, $p_n \in \mbox{Lip}\ \lambda_n$ and $q_n \in \mbox{Lip}\ \mu_n$. Also, let $f_1$ and $\hat{f}_1$ be the CHFIFs passing through $\Lambda$ and $\widehat{\Lambda}$ respectively. If $\Om = \Om_{n1} =1$, $\Gm = \Gm_{n2} =1$ and $\Th = \Th_{n3}=1$ for some $n1,n2,n3 \neq k$, then the upper bound of fractal dimension of CHFIF $\hat{f}_1$ is less than the upper bound of fractal dimension of CHFIF $f_1$  whereas  the lower bound of fractal dimension of CHFIF $\hat{f}_1$  is greater than the lower bound of fractal dimension of CHFIF $f_1$.
\end{theorem}
\begin{proof}
By Theorem~\ref{th:sm}, $f_1$ and $\hat{f}_1$ have same modulus of continuity i.e. $\om(\hat{f}_1; t) = \om(f_1; t) = O(|t|^{\de} \log |t|)$. It is shown in Theorem $4.1$ and $4.2$ of~\cite{chand07} that the fractal dimension of CHFIF $f_1$ satisfy
\begin{align*}
1-\frac{\log(\sum_{i=1}^N \ |\al_i|)}{\log |I_{max}|}\leq D_F(Graph(f_1)) \leq 1 - \de - \frac{\log N}{\log |I_{max}|}
\end{align*}
if $\Th = 1$ or $\Om =1$ and
\begin{align*}
1-\frac{\log(\sum_{i=1}^N \ |\gm_i|)}{\log |I_{max}|}\leq D_F(Graph(f_1)) \leq 1 - \de - \frac{\log N}{\log |I_{max}|}
\end{align*}
if $\Gm =1$. Similarly, following the lines of proof of Theorems $4.1$ and $4.2$ in~\cite{chand07}, the fractal dimension of CHFIF $\hat{f}_1$ satisfy
\begin{align*}
1-\frac{\log(\sum_{i=1}^N \ |\al_i|)}{\log |\hat{I}_{max}|}\leq D_F(Graph(\hat{f}_1)) \leq 1 - \de - \frac{\log (N+1)}{\log |\hat{I}_{max}|}
\end{align*}
if $\widehat{\Th} =1$ or $\widehat{\Om}  =1$ and
\begin{align*}
1-\frac{\log(\sum_{i=1}^N \ |\gm_i|)}{\log |\hat{I}_{max}|}\leq D_F(Graph(\hat{f}_1)) \leq 1 - \de - \frac{\log (N+1)}{\log |\hat{I}_{max}|}
\end{align*}
if $\widehat{\Gm} =1$. In the above two inequalities, $\hat{I}_{max} = \max\{|I_n|: n=1,\ldots, N, n\neq k, |\rho_x I_k|, |(1-\rho_x) I_k|\} \leq I_{max}$. Therefore,
the upper bound of fractal dimension of CHFIF $\hat{f}_1$ is less than the upper bound of fractal dimension of CHFIF $f_1$  whereas  the lower bound of fractal dimension of CHFIF $\hat{f}_1$  is greater than the lower bound of fractal dimension of CHFIF $f_1$.
\qed
\end{proof}

\section{Examples}\label{sec:five}

Consider a sample generalized interpolation data as $ \Delta = \{(0,0,10), (30,90,40), \\ (60,70,80), (100,20,30)\}$. The free variables $\al_n, \gm_n$ and constrained variables $\bt_n$ are chosen as given in Table~\ref{tab:Table1}. In Figure~\ref{fig:fig1}, Node-Node insertion is depicted by considering the node point $(45,60,20)$ in the given $\Delta$. Node-Knot insertion is shown in Figure~\ref{fig:fig2} by considering the point $(45,60,68.21)$ in the given $\Delta$. In Figure~\ref{fig:fig3}, the point $(45,198.43,20)$ is inserted in $\Delta$ to depict Knot-Node insertion and finally Figure~\ref{fig:fig4} shows the Knot-Knot insertion by considering the point $(45,198.43,68.21)$ in the given $\Delta$. In all the figures, the blue curve represents CHFIF $f_1$  obtained from the $\Delta$ while the black curve is CHFIF $\hat{f}_1$ obtained from  $\Delta \bigcup (\hat{x},\hat{y},\hat{z})$.

\begin{table}[!hbp]
\begin{center}\caption{Values of free variables and constrained variables
   \label{tab:Table1}} \vspace{0.2cm}
\begin{tabular}{|c|c|c|c|}  \hline  n & 1 & 2 & 3 \\ \hline
$\al_n$   & 0.2  &0.5   &0.3   \\  \hline
$\gm_n$   & 0.6  &0.2   &0.5   \\  \hline
$\bt_n$   & 0.3  &0.4   &0.1   \\  \hline
\end{tabular}
\end{center}
\end{table}

\begin{figure}[!htp]
\begin{center}\caption{Node Insertion} \vspace{0.2cm}
{\centering \hspace{1cm} \subfigure[ Node-Node insertion]{\epsfig{file=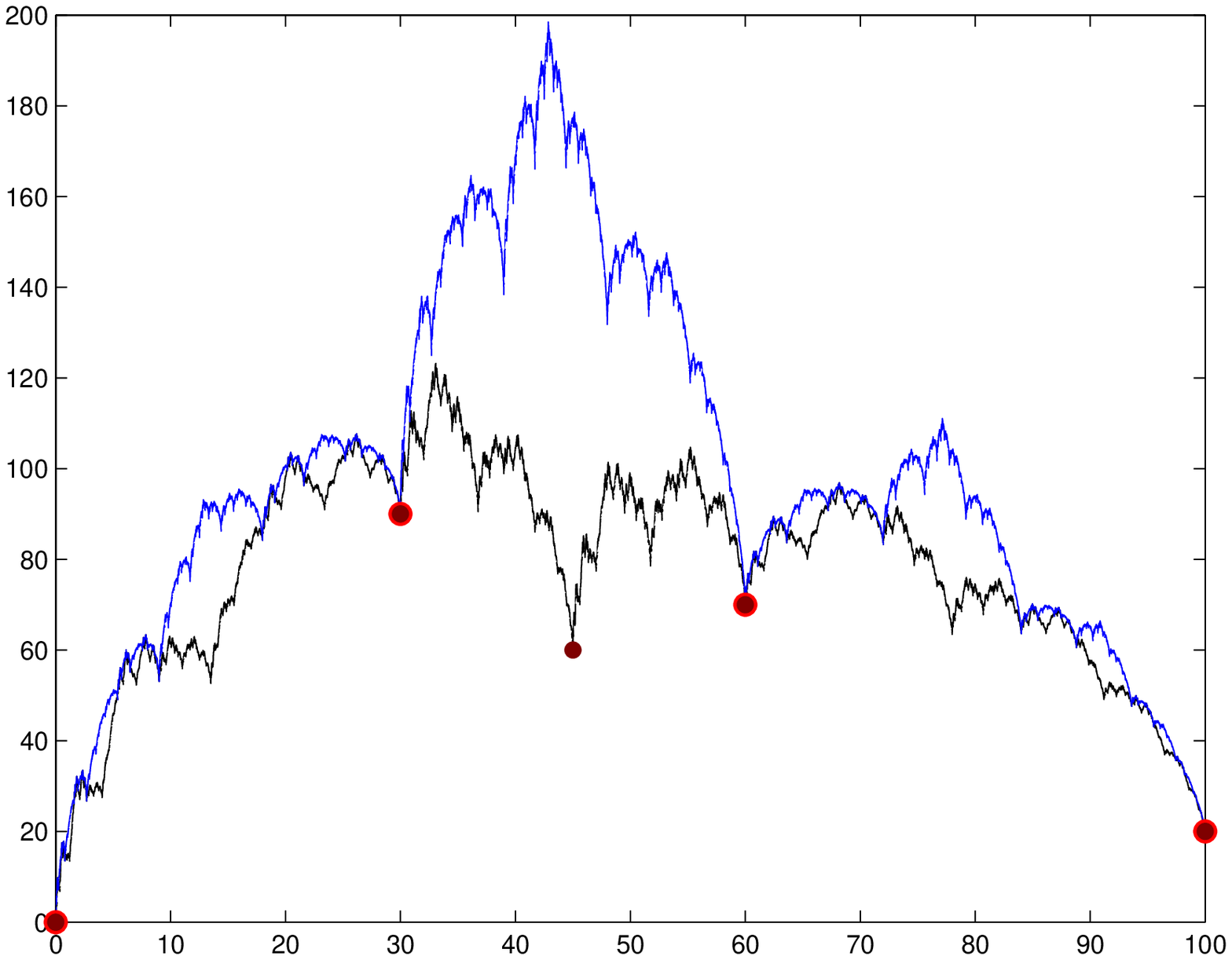, width=6cm}\label{fig:fig1}}
 \hspace{1cm}\subfigure[Node-Knot insertion]{\epsfig{file=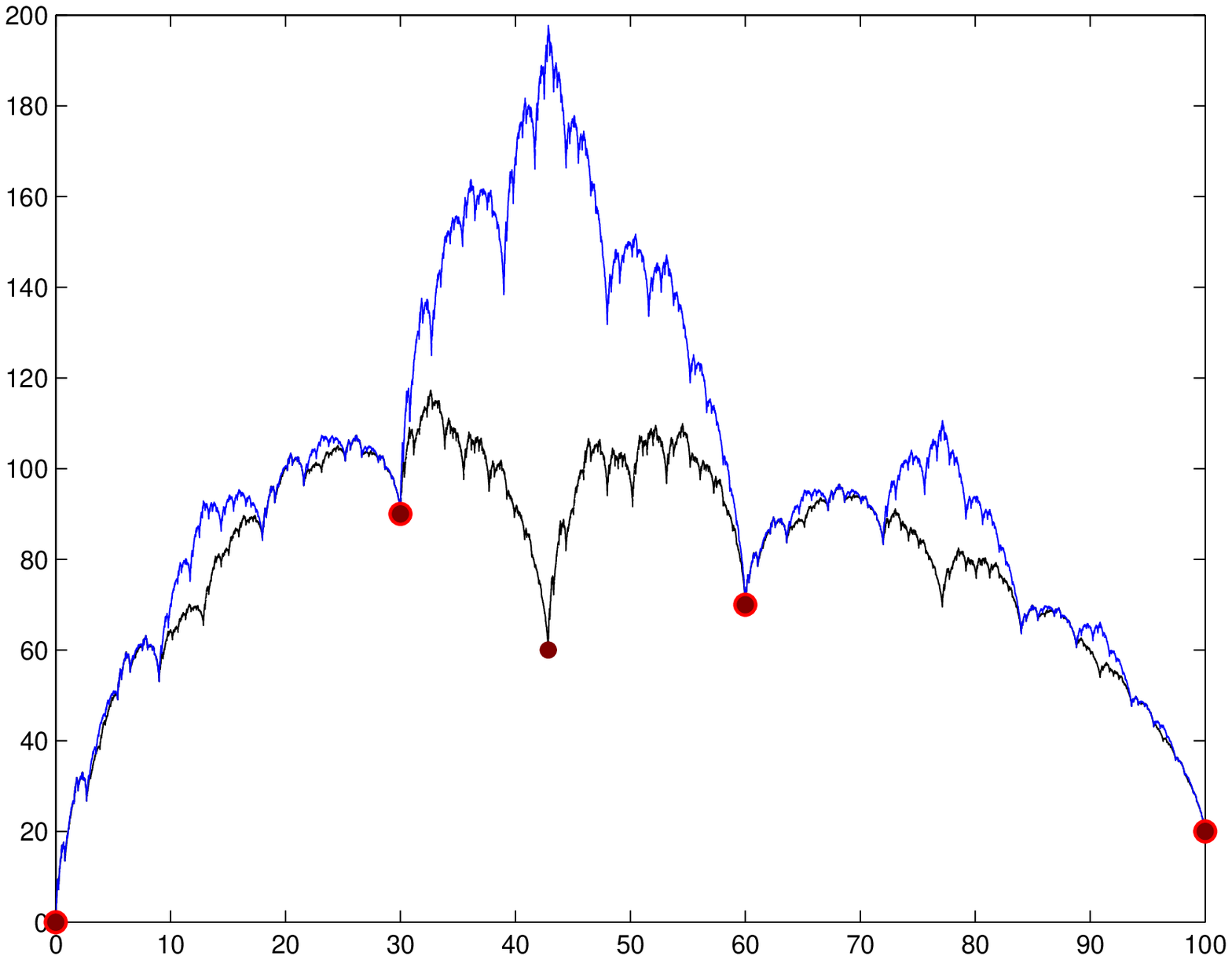, width=6cm}\label{fig:fig2}}}
  \\
{\centering  \hspace{1cm} \subfigure[ Knot-Node insertion]{\epsfig{file=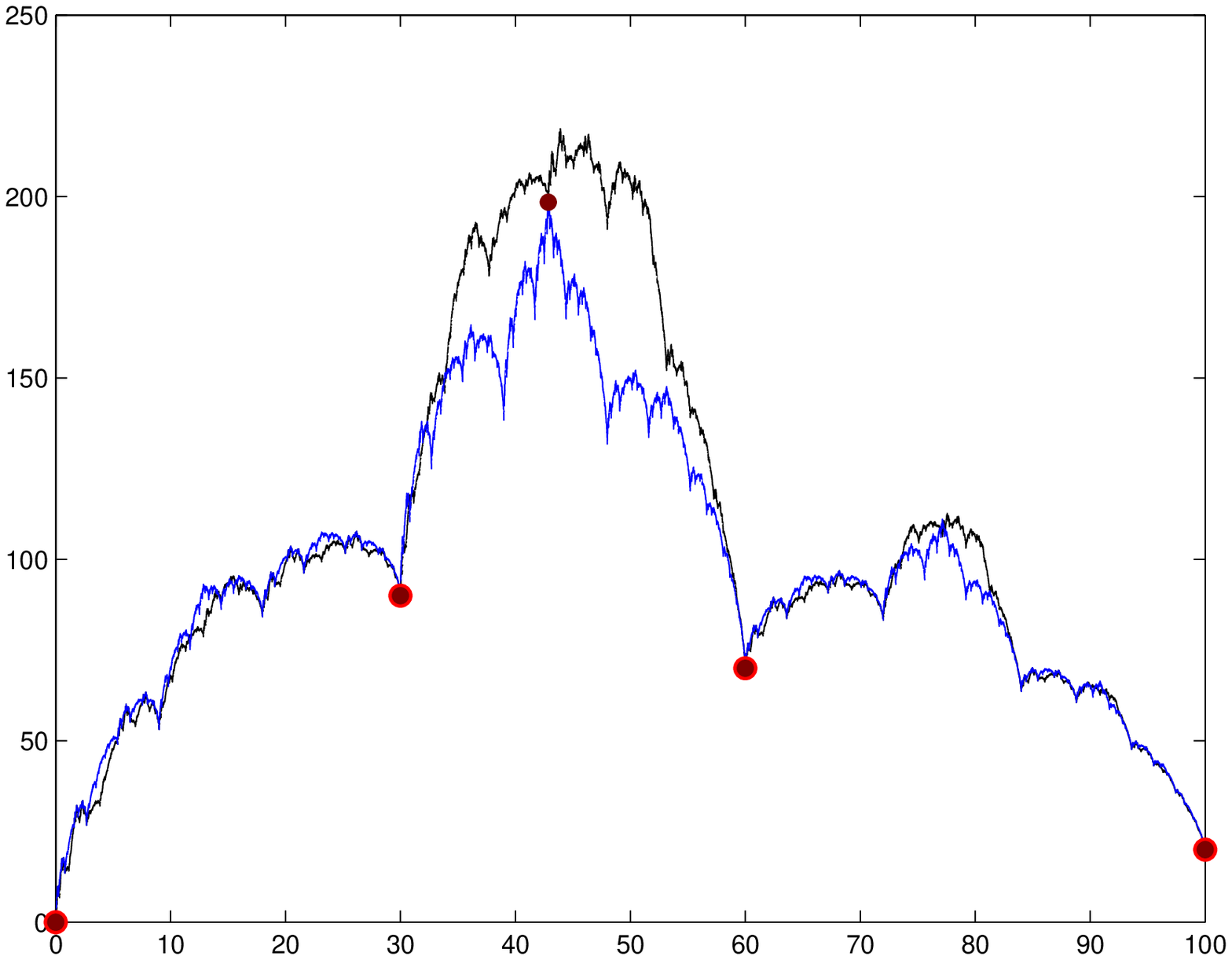, width=6cm}\label{fig:fig3}}
 \hspace{1cm}\subfigure[Knot-Knot insertion]{\epsfig{file=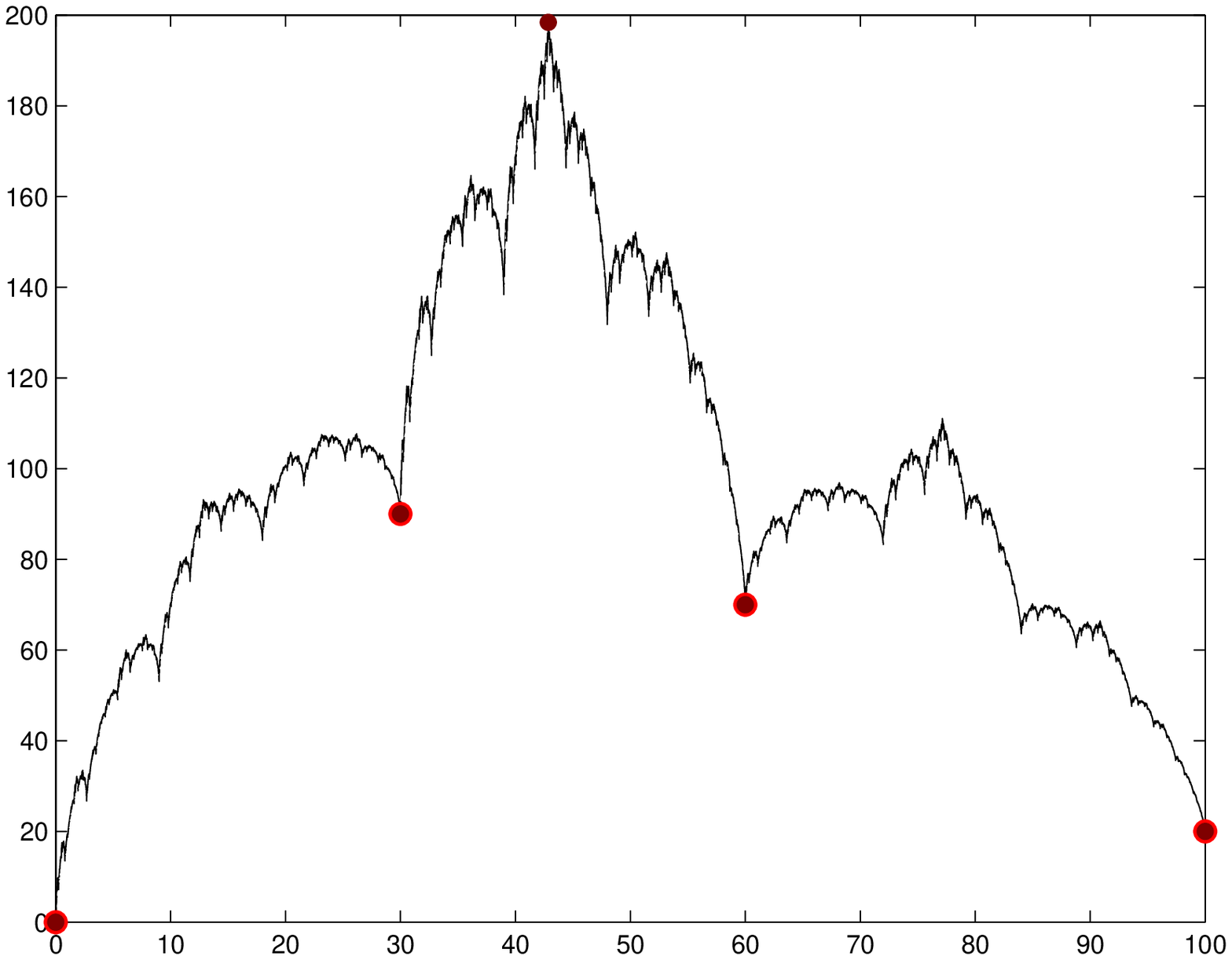, width=6cm}\label{fig:fig4}}}
\end{center}
\end{figure}

\end{document}